\documentclass[11pt]{amsart}
\usepackage{amsmath,amssymb,textcomp,xifthen,psfrag,graphicx,color}
\usepackage{amsfonts,amsthm}
\usepackage{color}

\usepackage{fullpage}
\usepackage[utf8]{inputenc}
\usepackage[T1]{fontenc}
\usepackage{hyperref}
\date{}

\usepackage{pgfplots}
\pgfplotsset{compat=1.18}
\usepgfplotslibrary{external}
\usepgfplotslibrary{colorbrewer}
\newtheorem{theorem}{Theorem}
\newtheorem{lemma}[theorem]{Lemma}
\newtheorem{cor}[theorem]{Corollary}

\theoremstyle{definition} 
\newtheorem{remark}[theorem]{Remark}

\newcommand{\dual}[2]{\langle#1\hspace*{.5mm},#2\rangle}
\newcommand{\vdual}[2]{(#1\hspace*{.5mm},#2)}

\newcommand{\diam}{\mathrm{diam}}
\newcommand{\wilde}{\widetilde}

\def\div{{\rm div}}

\newcommand{\trace}{ {\rm tr}}

\newcommand{\supp}{{\rm supp}}

\DeclareMathOperator*{\argmin}{arg\,min}
\DeclareMathOperator*{\ran}{ran}

\newcommand{\bR}{\ensuremath{\mathbb{R}}}
\newcommand{\bG}{\ensuremath{\mathbb{G}}}
\newcommand{\bN}{\ensuremath{\mathbb{N}}}

\newcommand{\cI}{\ensuremath{\mathcal{I}}}
\newcommand{\cJ}{\ensuremath{\mathcal{J}}}

\newcommand{\cN}{\ensuremath{\mathcal{N}}}

\newcommand{\fb}{\ensuremath{\mathbf{b}}}
\newcommand{\fg}{\ensuremath{\mathbf{g}}}

\newcommand{\cT}{\ensuremath{\mathcal{T}}}

\newcommand{\cS}{\ensuremath{\mathcal{S}}}

\newcommand{\cP}{\ensuremath{\mathcal{P}}}

\newcommand{\fu}{\ensuremath{\mathbf{u}}}

\newcommand{\fA}{\ensuremath{\mathbf{A}}}

\newcommand{\uu}{\boldsymbol{u}}

\newcommand{\n}{\boldsymbol{n}}

\begin{document}
\title{Space-time finite element interpolation of nonsmooth functions satisfying boundary conditions}
\date{\today}

\author{Thomas F\"uhrer}
\address{Facultad de Matem\'{a}ticas, Pontificia Universidad Cat\'{o}lica de Chile, Santiago, Chile}
\email{thfuhrer@uc.cl}

\author{Gregor Gantner}
\address{Faculty of Electrical Engineering, Mathematics and Computer Science, University of Twente, Enschede, Netherlands}
\email{gregor.gantner@utwente.nl}

\author{Roberto Gonz\'alez}
\address{Departamento de Matem\'atica, Universidad T\'ecnica Federico Santa Mar\'ia, Valpara\'iso, Chile}
\email{roberto.gonzalezf@sansano.usm.cl}

\author{Michael Karkulik}
\address{Departamento de Matem\'atica, Universidad T\'ecnica Federico Santa Mar\'ia, Valpara\'iso, Chile}
\email{michael.karkulik@usm.cl}

\thanks{{\bf Acknowledgment.} 
This work was supported by ANID through FONDECYT projects and 1250070 (TF) and 1250604 (MK and RG).
GG acknowledges funding by the Nederlandse Organisatie voor Wetenschappelijk Onderzoek (NWO, Dutch Research Council) through the Vidi project Optimal adaptive space-time boundary and finite element methods (OASTMethods) with file number VI.Vidi.243.148 under the grant https://doi.org/10.61686/LTBER75172.}

\begin{abstract}
  We construct Scott--Zhang-type space-time quasi-interpolation operators on tensor-product meshes for parabolic settings.
  Specifically, these are linear and uniformly bounded projections from the classical parabolic energy space
  onto a space-time tensor-product finite element space
  that preserve discrete Dirichlet boundary conditions on the lateral space-time boundary. We apply our operator in the context
  of space-time first-order system least-squares finite elements for parabolic equations to establish a quasi-optimal
  method for inhomogeneous Dirichlet boundary conditions.
\end{abstract}
\maketitle
\section{Introduction}
Motivated by the possibility of simultaneous adaptivity in
space and time and the availability of highly parallelizable solvers,
space-time variational formulations and associated numerical approximations for time-dependent PDEs
have been considered for at least fifty years, cf.~the pioneering work~\cite{ArgyrisS_69}.
But while the analysis of classical time-stepping methods is by now textbook knowledge, cf.~\cite{Thomee_06},
the mathematical theory of space-time finite elements has only recently been subject to a more thorough investigation.
A central topic in the analysis of finite elements is the interpolation of non-smooth functions.
As solutions to variational formulations of PDE lack the necessary smoothness for point evaluation,
more advanced interpolation operators have to be built. These operators are referred to as \textit{quasi-interpolation operators}.
In the context of $h$-version finite elements with fixed polynomial degree,
the foundational construction given by Cl\'ement~\cite{Clement_75} proceeds by local averaging on element patches, and is hence defined on $L^2$.
This approach was extended in various directions to account for curvilinear elements, Hermite elements, common vector-valued finite element spaces, and $hp$-methods,
cf.~\cite{Bernardi_89,BernardiGirault_98,GiraultScott_02,ErnGuermond_17,Melenk_05},
and often with a view to residual a-posteriori error estimation, cf.~\cite{Carstensen_99,Carstensen_06}.
Arguably the most ubiquitous quasi-interpolation operator is due to Scott and Zhang~\cite{ScottZhang_90}. This operator is defined
by local averaging against dual basis functions, and therefore turns out to be a projection onto the finite element space.
Furthermore, degrees of freedom on the boundary of the domain are determined by averaging on the boundary itself,
and therefore this operator reproduces discrete boundary values.
Among many others, applications of quasi-interpolation operators with such properties include
convergence analysis of adaptive finite element methods~\cite{CarstensenFPP_14} and
the incorporation of inhomogenous Dirichlet boundary conditions into finite element schemes~\cite{AuradaFKPP_13}, cf. Section~\ref{sec:ihom} below.

The approach of Scott and Zhang to average against dual basis functions can be extended to dual spaces, i.e., Sobolev spaces of negative order.
Corresponding works are~\cite{Fuehrer_21,DieningST_23,Tantardini_13}, where projection operators defined on $H^{-1}(\Omega)$
for polyhedral domains $\Omega\subset\bR^d$ mapping into finite element spaces are constructed. Those finite element spaces can either be equipped with
homogeneous boundary conditions or not.
Applications include regularization of rough data in PDEs~\cite{Fuehrer_24,FuehrerHK_22},
multilevel decompositions and operator preconditioning~\cite{BramblePV_00,WuZ_17,Fuehrer_21,StevensonvV_20,StevensonvV_21},
and, the main topic of the work at hand, space-time finite element methods for parabolic problems \cite{FuehrerK_21,GantnerS_21,GantnerS_24}.
The energy space for parabolic PDEs posed on $J\times \Omega$ with a time interval $J$ and a spatial domain $\Omega\subset\bR^d$ is given by
$L^2(J;H^1(\Omega))\cap H^1(J;H^{-1}(\Omega))$. The latter space allows for the application of the aforementioned
spatial operators defined on $H^{-1}(\Omega)$ pointwise in time.
Coupled with a quasi-interpolation operator in time, this gives a
locally defined and locally stable projection operator on discrete spaces on tensor-product meshes in space-time,~\cite{DieningST_23,StevensonS_23}.
If the spatial operators map into a finite element space with homogeneous boundary conditions, then the space-time interpolation operator
exhibits homogeneous boundary conditions on the lateral space-time boundary. If the spatial operators map into a finite element space
without boundary conditions, then so will the space-time interpolation operator. In the latter case, discrete boundary conditions on the lateral
space-time boundary are in general not preserved by the space-time interpolation operator.
The second key property of Scott and Zhang's original construction --- the reproduction of 
discrete boundary values --- will be provided in the present work.
We will follow the original construction of Scott and Zhang for the elliptic case, and define degrees of freedom
of the space-time quasi-interpolation operator on the boundary by averaging against a corresponding dual basis of functions living on the
lateral space-time boundary $J\times\partial\Omega$. The boundedness of Scott--Zhang-type operators in $L^2$-based spaces is straight-forward
and uses the Cauchy--Schwarz inequality and properties of the dual basis functions.
The boundedness of derivatives, on the other hand, is shown by applying
inverse estimates to the finite element basis functions. These inverse estimates produce negative powers of the mesh-size, which have to be
accounted for by using Poincar\'e inequalities. These Poincar\'e inequalities require that the operator reproduces constants locally.
In the present case of a parabolic setting, the temporal derivatives enter the norm of interest via the term $H^1(J;H^{-1}(\Omega))$.
Hence, special parabolic Poincar\'e estimates have to be used to account for the anisotropy of regularities. It turns out that the uniform boundedness
in the parabolic energy space requires a parabolic scaling.

A common practice to incorporate inhomogeneous Dirichlet boundary conditions into finite element methods is to extend the boundary conditions
into the interior. The remaining part of the solution then fulfills homogeneous Dirichlet boundary conditions and can be computed by
suitable methods. While those methods usually are quasi-optimal, it is not obvious that this quasi-optimality carries over to
the resulting method for inhomogeneous boundary conditions. It turns out that the existence of a suitable Scott--Zhang-type operator is a sufficient
conditions for quasi-optimality,~\cite{AuradaFKPP_13,Stevenson_24}. In the case of space-time first-order system least-squares finite elements for
parabolic equations, we will apply our operator to prove such a result.
\section{Inhomogeneous Boundary Conditions}\label{sec:ihom}
An immediate application of Scott--Zhang-type operators is the incorporation of inhomogeneous Dirichlet data into numerical
schemes tailored to homogeneous boundary conditions.
In the special case of Lagrange finite elements for the Laplace equation, the next result can be found in~\cite[Prop.~2.3]{AuradaFKPP_13}.
We closely follow~\cite{Stevenson_24} for a more abstract presentation of this topic. In the following Lemma, it is instructive to
think of $F$ as a PDE operator, of $T$ as a trace operator, and of $P_h$ as an operator approximating Dirichlet traces.
\begin{lemma}\label{lemma:ihom}
  Let $X,Y_1,Y_2$ be Hilbert spaces and $F:X\to Y_1$ and $T:X\to Y_2$ linear and bounded.
  \begin{enumerate}
    \item[(i)] Set $X_0 := \ker T$ and suppose that $F_0 := F|_{X_0}:X_0\to Y_1$ is boundedly invertible and that $T$ is surjective.
      Then, $G:=(F,T):X\to Y:=Y_1\times Y_2$ is boundedly invertible.
    \item[(ii)] In addition to the assumptions of \rm{(i)},
      let $X_h$ be a finite-dimensional subspace of $X$ and $X_{0,h}:=X_h\cap X_0$.
      Suppose that we have a numerical method
      $F_{0,h}^{-1}:Y_1\to X_{0,h}$ such that $\bG_{0,h}:X\to X_{0,h}$ defined by $\bG_{0,h} := F_{0,h}^{-1}F$ is
      a linear and uniformly bounded projection. Furthermore, suppose that $P_h:Y_2 \to \ran T|_{X_h}$ is a linear and uniformly bounded
      projection and $E_h:\ran T|_{X_h}\to X_h$ a (not necessarily linear and/or bounded) extension operator, i.e., $TE_h$ is the identity.
      Then, the operator $\bG_h:X\to X_h$ defined by
      \begin{align*}
	\bG_h := E_hP_hT + \bG_{0,h} (1-E_hP_h T)
      \end{align*}
      is independent of the choice of the operator $E_h$.
    \item[(iii)] In additions to the assumptions of \rm{(i)} and \rm{(ii)},
      suppose that $J_h:X\to X_h$ is a linear and uniformly bounded projection which fulfills
      \begin{align}\label{eq:cond1}
	\forall v \in X:\quad T v \in \ran T|_{X_h} \quad\Longrightarrow\quad T J_h v = Tv.
      \end{align}
      Then, $\bG_h$ is quasi-optimal, i.e.,
      \begin{align*}
	\forall u\in X:\quad \| u - \bG_{h} u \|_X \lesssim \inf_{w_h\in X_{h}} \| u - w_h \|_X.
      \end{align*}
  \end{enumerate}
\end{lemma}
\begin{proof}
  Statement (i) can be found in~\cite[Lemma~2.7]{GantnerS_21}. To prove statement (ii), note first that the given assumptions imply
  that $\bG_{0,h}$ is quasi-optimal, that is
  \begin{align}\label{thm:ihom:eq1}
    \forall u\in X:\quad \| u - \bG_{0,h} u \|_X \leq (1+\| \bG_{0,h} \|) \inf_{w_h\in X_{0,h}} \| u - w_h \|_X.
  \end{align}
  Now, the fact that $\bG_h$ does not depend on the particular choice of the operator $E_h$ can be seen as follows.
  Let $\wilde E_h:\ran T|_{X_h}\to X_h$ be another extension operator.
  Then, $(\wilde E_h-E_h)P_hT:X\to X_{0,h}$. Using that $\bG_{0,h}$ is a projection onto $X_{0,h}$ gives
  \begin{align*}
    \bG_{0,h} (\wilde E_h-E_h)P_hT = (\wilde E_h-E_h)P_hT,
  \end{align*}
  or, put otherwise,
  \begin{align*}
    E_hP_hT + \bG_{0,h} (1-E_hP_h T) = \wilde E_hP_hT + \bG_{0,h} (1-\wilde E_hP_h T).
  \end{align*}
  To conclude statement (iii), note first that
  \begin{align*}
    \bG_{0,h} F_0^{-1}F(1-E_hP_hT) = F_{0,h}^{-1}FF_0^{-1}F(1-E_hP_hT) = \bG_{0,h} (1-E_hP_hT).
  \end{align*}
  The triangle inequality and quasi-optimality~\eqref{thm:ihom:eq1} then give
  \begin{align*}
    \| u - \bG_h u \|_X
    &\leq \| u - E_hP_hTu - F_0^{-1}F(1-E_hP_hT)u \|_X
    + \| F_0^{-1}F(1-E_hP_hT)u - \bG_{0,h}(1-E_hP_hT)u \|_X\\
    &= \| u - E_hP_hTu - F_0^{-1}F(1-E_hP_hT)u \|_X
    + \| F_0^{-1}F(1-E_hP_hT)u - \bG_{0,h}F_{0}^{-1}F(1-E_hP_hT)u \|_X\\
    &\leq \| u - E_hP_hTu - F_0^{-1}F(1-E_hP_hT)u \|_X + (1+\|\bG_{0,h}\|) \inf_{w_h\in X_{0,h}} \| F_0^{-1}F(1-E_hP_hT)u - w_h \|_X.
  \end{align*}
  In order to treat the last term on the right-hand side, for $w_h\in X_{0,h}$ we estimate
  \begin{align*}
    \| F_0^{-1}F(1-E_hP_hT)u - w_h \|_X \leq \| u - E_hP_hTu - F_0^{-1}F(1-E_hP_hT)u \|_X + \| u - E_hP_hTu - w_h \|_X,
  \end{align*}
  and, hence,
  \begin{align*}
    \| u - \bG_h u \|_X &\leq (2+\|\bG_{0,h}\|) \| u - E_hP_hTu - F_0^{-1}F(1-E_hP_hT)u \|_X +
    (1+\|\bG_{0,h}\|)\inf_{\substack{\wilde w_h\in X_h\\T\wilde w_h = P_hTu}} \|u-\wilde w_h\|_X.
  \end{align*}
  Next, as $FF_0^{-1}=\mathrm{id}$, $TF_0^{-1}=0$, and $TE_h=\mathrm{id}$, we have
  \begin{align*}
    F\left( (1-E_hP_hT) - F_0^{-1}F(1-E_hP_hT) \right) &= 0,\\
    T\left( (1-E_hP_hT) - F_0^{-1}F(1-E_hP_hT) \right) &= (1-P_h)T,
  \end{align*}
  and conclude
  \begin{align*}
    \| u - \bG_h u \|_X \leq ( 2+\|\bG_{0,h}\|) \| G^{-1} \| \| (1-P_h)Tu \|_{Y_2} +
    (1+\|\bG_{0,h}\|) \inf_{\substack{\wilde w_h\in X_h\\T\wilde w_h = P_hTu}} \|u-\wilde w_h\|_X.
  \end{align*}
  The operator $E:Y_2\to X$ defined by $E(\cdot):= G^{-1}(0,\cdot)$ is a linear and bounded right-inverse of $T$.
  Note that~\eqref{eq:cond1} implies 
  \begin{align*}
    T \left( J_h u -J_h E(Tu-P_hTu) \right) = P_hTu,
  \end{align*}
  and, hence,
  \begin{align*}
    \inf_{\substack{\wilde w_h\in X_h\\T\wilde w_h = P_hTu}} \|u-\wilde w_h\|_X &\leq \| u - J_h u +J_h E(Tu-P_hTu) \|_X.
  \end{align*}
  Then, using that $J_h$ is a projection onto $X_h$, we obtain
  \begin{align*}
    \| u - \bG_h u \|_X \leq ( 2+\|\bG_{0,h}\| ) \left[ \left( \| G^{-1} \| + \|J_h\|\|E\| \right)\| (1-P_h)Tu \|_{Y_2} +(1+\|J_h\|)\inf_{w_h\in X_h} \| u - w_h \|_X \right].
  \end{align*}
  Finally, as $P_h$ is a projection, we obtain the estimate
  \begin{align*}
    \| (1-P_h)Tu \|_{Y_2} \leq \|1-P_h\|\|T\| \inf_{w_h\in X_h} \|u-w_h\|_X,
  \end{align*}
  and, consequently,
  \begin{align*}
    \| u - \bG_h u \|_X \leq ( 2+\|\bG_{0,h}\| ) \left[ \left( \| G^{-1} \| + \|J_h\|\|E\| \right)\|1-P_h\|\|T\| +(1+\|J_h\|) \right]\inf_{w_h\in X_h} \| u - w_h \|_X.
  \end{align*}
\end{proof}
\section{Space-time Scott--Zhang interpolation operators}
\subsection{Mathematical setting and main result}
Let $\Omega\subset\bR^d$ be a polyhedral Lipschitz domain, $J=(0,T)$ a finite time interval,
and $Q:= J\times \Omega\subset \bR^{1+d}$ the space-time cylinder.
We denote the spatial boundary by $\Gamma:=\partial\Omega$ and the lateral space-time boundary by $\Sigma:= J\times\Gamma$.
The differential operator $\partial_t$ denotes the temporal derivative, while $\nabla_x$ denotes the gradient with respect to the
spatial variables.
We will use standard notation for Lebesgue and Sobolev spaces on $\Omega$ or $\Gamma$, for example $L^2(D) = H^0(D), H^1(D), H^1_0(D)$, for subsets
$D$ of $\Omega$ or $\Gamma$,
or the fractional trace space $H^{1/2}(\Gamma)$.
Additionally, for open subintervals $I\subset J$, we write $L^2(I;X)$ for the Bochner space of all measurable and square-integrable functions mapping from
the time interval $I$ to some Banach space $X$.
For $r\in(0,1]$, we also use Bochner--Sobolev spaces $H^r(I;X) := \left\{ v\in L^2(I;X) \mid \| v \|_{H^r(I;X)}<\infty \right\}$, where
\begin{align*}
  \| v \|_{H^r(I;X)}^2 &:= \| v \|_{L^2(I;X)}^2 + | v |_{H^r(I;X)}^2,\\
  | v |_{H^r(I;X)}^2 &:=
  \begin{cases}
    \int_I\int_I \frac{\| v(s)-v(t) \|_{X}^2}{|s-t|^{1+2r}}\,ds\,dt & r\in(0,1),\\
    \| \partial_t v \|_{L^2(I;X)}^2 & r=1.
  \end{cases}
\end{align*}

Let $\cT$ be a conforming simplicial partition of a domain or manifold $D\subset\bR^n$, $n\in\bN$ (by conforming we mean that there are no
hanging nodes).
For $p\in\bN$, $\cS^p(\cT)$ denotes the space of globally continuous $\cT$-piecewise polynomials of degree at most $p$.
The notation $\cS^p_0(\cT)$ indicates the subspace of functions vanishing on the boundary of $D$.
For $K\in\cT$, we let $\omega_K\subseteq\cT$ be the (relatively) open set covered by the simplices sharing at least a node with $K$,
\begin{align*}
  \overline{\omega_K} := \bigcup_{K' \text{ sharing a node with } K} \overline{K'}.
\end{align*}
\medskip

The energy space for solutions to parabolic PDEs posed on the space-time cylinder $Q$ is
\begin{align*}
  X := L^2(J;H^1(\Omega))\cap H^1(J;H^{-1}(\Omega)).
\end{align*}

Traces of functions in $X$ on the lateral space-time boundary $\Sigma$ belong to
\begin{align*}
 H^{1/2,1/4}(\Sigma):=
 L^2(J;H^{1/2}(\Gamma)) \cap H^{1/4}(J;L^2(\Gamma))
\end{align*}
cf.~\cite{LionsMagenes_72,Costabel_90},
and the corresponding linear and bounded trace operator will be denoted by $\trace_\Sigma$.
\bigskip

We use partitions $\cT_{t,x}$ of $Q$ into time-space prisms $K=K_t\times K_x$, where $K_t\subset\bR$ is a time interval and
$K_x\subset\bR^d$ a $d$-simplex. As in~\cite{GantnerS_24}, we assume that $\cT_{t,x}$ is of tensor-product form,
that is $\cT_{t,x} = \cT_t\otimes\cT_x$, where $\cT_x$ is a conforming simplicial partition of $\Omega$,
and $\cT_t$ is a partition of $J$ into intervals. For intervals $K_t$ and simplices $K_x$, we
denote by $h_{K_t}$ and $h_{K_x}$ their diameter. Furthermore, $h_x\in L^\infty(\Omega)$ is defined by $h_x|_{K_x}=h_{K_x}$ for all $K_x\in\cT$,
and $h_t\in L^\infty(J)$ is defined analogously,
For polynomial degrees $\ell,k\in\bN$, we define tensor-product spaces
\begin{align*}
  \cS^{k,\ell}(\cT_{t,x}) := \cS^k(\cT_t)\otimes\cS^\ell(\cT_x),\\
  \cS^{k,\ell}_0(\cT_{t,x}) := \cS^k(\cT_t)\otimes\cS^\ell_0(\cT_x).
\end{align*}
In particular, a function $u_h\in \cS^{k,\ell}(\cT_{t,x})$ locally has the form $u_h|_K \in \cP^k(K_t)\otimes \cP^\ell(K_x)$.
Our main result is the following.
\begin{theorem}\label{thm:main}
  Let $\cT_t$ be a partition of $J$ into intervals and $\cT_x$ be conforming simplicial partition of $\Omega$.
  Then, there exists a linear operator
  \begin{align*}
    \cJ_{\cS^{k,\ell}(\cT_{t,x})}: L^2(J;H^1(\Omega)) \cap H^1(J;H^{-1}(\Omega))\to \cS^{k,\ell}(\cT_{t,x})
  \end{align*}
  which
  \begin{enumerate}
    \item[(i)] \textbf{is a projection} onto $\cS^{k,\ell}(\cT_{t,x})$, that is,
      \begin{align*}
	v \in \cS^{k,\ell}(\cT_{t,x}) \quad\Longrightarrow\quad \cJ_{\cS^{k,\ell}(\cT_{t,x})} v = v,
      \end{align*}
    \item[(ii)] \textbf{reproduces discrete boundary conditions}, that is,
      \begin{align*}
	\forall v\in L^2(J;H^1(\Omega)) \cap H^1(J;H^{-1}(\Omega)):\quad \trace_\Sigma v\in \ran\trace_\Sigma|_{\cS^{k,\ell}(\cT_{t,x})}
	\quad\Longrightarrow\quad \trace_\Sigma \cJ_{\cS^{k,\ell}(\cT_{t,x})} v=\trace_\Sigma v,
      \end{align*}
    \item[(iii)] \textbf{fulfills the local stability estimates}
      \begin{align*}
	\| \partial_t \cJ_{\cS^{k,\ell}(\cT_{t,x})} v \|_{ L^2(K_t;H^{-1}(\Omega))}
	&\lesssim \| \partial_tv \|_{L^2(\omega_{K_t};H^{-1}(\Omega))} + h_{K_t}^{-1} \| h_{x}^2 \nabla_x v \|_{L^2(\omega_{K_t}\times \Omega)}\\
	\| \cJ_{\cS^{k,\ell}(\cT_{t,x})} v \|_{L^2(K_t\times K_x)}
        &\lesssim
        \| v \|_{L^2(\omega_{K_t}\times\omega_{K_x})}
	+ h_{K_x}\| v \|_{L^2(\omega_{K_t}\times\omega_{K_x})}\\
        \| \nabla_x \cJ_{\cS^{k,\ell}(\cT_{t,x})} v \|_{L^2(K_t\times K_x)}
        &\lesssim \| \nabla_x v \|_{L^2(\omega_{K_t}\times\omega_{K_x})}
	+ h_{K_t} h_{K_x}^{-2} \| \partial_t v \|_{L^2(\omega_{K_t};H^{-1}(\omega_{K_x}))},
      \end{align*}
      for all $v \in L^2(J;H^1(\Omega)) \cap H^1(J;H^{-1}(\Omega))$ and $K_t\times K_x\in\cT_{t,x}$,
    with hidden constants depending only on local quasi-uniformity of $\cT_t$ and shape-regularity (or local quasi-uniformity
    for $d=1$) of $\cT_x$, and the polynomial degrees $k$ and $\ell$.
  \end{enumerate}
\end{theorem}
\begin{remark}\label{rem1}
  Considering uniform meshes $\cT_t$ and $\cT_x$, the equal scaling $h_x\eqsim h_t$ already implies that $h_x^2\lesssim h_t$, which in turn
  yields $h_{t}^{-1} h_x^{2} \lesssim 1$.
  On the other hand, $h_t\lesssim h_x^2$ is needed to ensure
  $h_{t}h_{x}^{-2}\lesssim 1$. Eventually, the parabolic scaling $h_t\eqsim h_x^2$ gives robust
  stability estimates in the preceding theorem. For later reference, we mention in particular that parabolic scaling $h_t\simeq h_x^2$ gives
  the global stability estimate
  \begin{align*}
    \| \cJ_{\cS^{k,\ell}(\cT_{t,x})} v \|_{L^2(J;H^1(\Omega))\cap H^1(J;H^{-1}(\Omega))}
    \lesssim 
    \| v \|_{L^2(J;H^1(\Omega))\cap H^1(J;H^{-1}(\Omega))}
  \end{align*}
  by summing the stability estimates of Theorem~\ref{thm:main}.
  \qed
\end{remark}
\subsection{Application to inhomogeneous boundary conditions in space-time discretizations}\label{sec:appl}
In this section, we show how to apply Theorem~\ref{thm:main} in the context of space-time first-order system least squares (FOSLS)
finite elements for parabolic equations, employing notation as in Lemma~\ref{lemma:ihom}.
We consider the parabolic initial/boundary value problem to find $u$ such that
\begin{equation}\label{eq:par}
\begin{aligned}
    \partial_t u - \div_x\;\fA\nabla_x u + \fb\cdot\nabla_x u + cu &= f + \div_x\;\fg &&\text{ on } Q,\\
    u &= u_D &&\text{ on } \Sigma,\\
    u &= u_0. &&\text{ on } \left\{ 0 \right\}\times\Omega,
\end{aligned}
\end{equation}
where $\fA=\fA^\top\in L^\infty(Q)^{d\times d}$ uniformly positive definite, $\fb\in L^\infty(Q)^d$, $c\in L^\infty(Q)$,
and $f\in L^2(Q)$, $\fg\in L^2(Q)^d$, $u_D\in H^{1/2,1/4}(\Sigma)$, and $u_0\in L^2(\Omega)$.
Following~\cite{FuehrerK_21,GantnerS_21}, we set
\begin{align*}
  X := \left\{ \fu=(u_1,\fu_2)\in L^2(J;H^1(\Omega))\times L^2(Q)^{d} \mid \div\;\fu \in L^2(Q) \right\}
\end{align*}
with graph norm
\begin{align*}
  \| \fu \|_X^2 := \| u_1 \|_{L^2(J;H^1(\Omega))}^2 + \| \fu_2 \|_{L^2(Q)^d}^2 + \| \div\;\fu \|_{L^2(Q)}^2.
\end{align*}
Here, $\div\,(u_1,\fu_2)= \partial_t u_1 + \div_x\,\fu_2$ is the time-space diverence operator.
Then, $X$ is a Hilbert space. Set
\begin{align*}
  Y_1 &:= L^2(Q)^d\times L^2(Q) \times L^2(\Omega),\\
  Y_2 &:= H^{1/2,1/4}(\Sigma).
\end{align*}
Define $F:X\to Y_1$, $T:X\to Y_2$ by
\begin{align*}
  F \uu &= F(u_1,\fu_2) :=( \fu_2+\fA\nabla_x u_1, \div\;\uu-\fb\fA^{-1}\fu_2 + cu_1, u_1(0) ),\\
  T \uu &:= \trace_\Sigma u_1.
\end{align*}
The above objects fulfill the conditions of Lemma~\ref{lemma:ihom} (i), cf.~\cite[Thm.~2.8]{GantnerS_21}. Hence, the equation
$G \uu = (\fg,f,u_0,u_D)$ has a unique solution, which connects to~\eqref{eq:par} via $\uu = (u,-\fA \nabla_x u)$.
Let $\cT_t$ be a partition of $J$ into intervals and $\cT_x$ be conforming simplicial partition of $\Omega$.
Following~\cite{GantnerS_24}, we set
\begin{align*}
  X_{h} &:= \cS^{k,\ell}(\cT_{t,x}) \times [ \cP^{k-1}(\cT_t)\otimes RT^\ell(\cT_x) ],
\end{align*}
where $\cP^{k-1}(\cT_t)$ is the space of $\cT_t$-piecewise polynomials of degree $k-1$, and
$RT^\ell(\cT_x)$ is the Raviart--Thomas space of degree $\ell$.
Then $X_{0,h} = \cS^{k,\ell}_0(\cT_{t,x}) \times [ \cP^{k-1}(\cT_t)\otimes RT^\ell(\cT_x) ]$.
The numerical method for homogeneous boundary conditions $F_{0,h}^{-1}:Y_1\to X_{0,h}$ is given by FOSLS finite elements, i.e.,
\begin{align*}
  F_{0,h}^{-1}(\fg,f,u_0) = \argmin_{\uu_h = (u_{1,h},\uu_{2,h}) \in X_{0,h}} &\| \fu_{2,h}+\fA\nabla_x u_{1,h} - \fg \|_{L^2(Q)^d}^2\\
  \qquad &+ \| \div\;\uu_h-\fb\fA^{-1}\fu_{2,h} + cu_{1,h} - f \|_{L^2(Q)}^2 + \| u_{1,h}(0) - u_0 \|_{L^2(\Omega)}^2.
\end{align*}
The corresponding approximation operator $\bG_{0,h}$ is linear and uniformly bounded, as $F$ and $F_{0,h}^{-1}$ are.
As operator $P_h$ we may take any linear projection which is bounded in $H^{1/2,1/4}(\Sigma)$. Candidates are Scott--Zhang-type
operators on the boundary, or the
$L^2(\Sigma)$-orthogonal projection onto $\cS^k(\cT_t)\otimes\cS^\ell(\cT_x^\Gamma)$. As discrete extension operator $E_h$,
we conveniently set
\begin{align*}
  E_h(\trace_\Sigma u_{1,h}) := (\wilde u_{1,h},\bf0),
\end{align*}
where $\wilde u_{1,h}$ is obtained by setting inner degrees of freedom (defined as point evaluations at Lagrange nodes) to $0$.
The above objects fulfill the conditions of Lemma~\ref{lemma:ihom} (ii).
In order to apply point (iii) of Lemma~\ref{lemma:ihom} and conclude quasi-optimality of the corresponding operator $\bG_h$,
it remains to define a linear and bounded projection
$J_h:X\to X_h$ which fulfills~\eqref{eq:cond1}. To this end, we will apply our operator $\cJ_{\cS^{k,\ell}(\cT_{t,x})}$ in a corresponding
construction already used in~\cite[Sec.~5.2]{StevensonS_23}.
\begin{lemma}\label{lem:1}
  Let $\cT_t$ be a partition of $J$ into intervals and $\cT_x$ be conforming simplicial partition of $\Omega$.
  Then, there exists a linear projection
  \begin{align*}
    \cI_{\cP^{k-1}(\cT_t)\otimes RT^\ell(\cT_x)}:L^2(Q)^d \to \cP^{k-1}(\cT_t)\otimes RT^\ell(\cT_x)
  \end{align*}
  which is bounded in $L^2(Q)^d$ such that the operator $J_h:X\to X_h$
  \begin{align*}
    J_h\fu = J_h(u_1,\fu_2) := \left( \cJ_{\cS^{k,\ell}(\cT_{t,x})}u_1,
    \cI_{\cP^{k-1}(\cT_t)\otimes RT^\ell(\cT_x)}
    \bigl( 
      \fu_2-\nabla_x(-\Delta_x)^{-1}\partial_t(u_1-\cJ_{\cS^{k,\ell}(\cT_{t,x})}u_1)
    \bigr)
\right)
  \end{align*}
  fulfills
  \begin{align*}
    \div\, J_h\fu
    = \Pi_{\cP^{k-1}(\cT_{t})\otimes\cP^\ell(\cT_x)}\div\,\fu,
  \end{align*}
  where $\Pi_{\cP^{k-1}(\cT_{t})\otimes\cP^\ell(\cT_x)}:L^2(Q)\to \cP^{k-1}(\cT_{t})\otimes\cP^\ell(\cT_x)$ is the $L^2(Q)$-orthogonal projection.
\end{lemma}
\begin{cor}
  Let $\cT_t$ be a partition of $J$ into intervals and $\cT_x$ be conforming simplicial partition of $\Omega$.
  Suppose the parabolic scaling $h_t\eqsim h_x^2$. Then, the
  linear operator $J_h:X\to X_h$ from Lemma~\ref{lem:1} fulfills the conditions of Lemma~\ref{lemma:ihom} (iii). That is,
  $J_h$ is a uniformly bounded projection, and
  \begin{align*}
    \forall v \in X:\quad T v \in \ran T|_{X_h} \quad\Longrightarrow\quad T J_h v = Tv.
  \end{align*}
  Consequently, the operator $\bG_h$ is quasi-optimal.
\end{cor}
\begin{proof}
  As $\cJ_{\cS^{k,\ell}(\cT_{t,x})}$ is a projection onto $\cS^{k,\ell}(\cT_x)$ and $\cI_{\cP^{k-1}(\cT_t)\otimes RT^\ell(\cT_x)}$ is a projection
  onto $\cP^{k-1}(\cT_t)\otimes RT^\ell(\cT_x)$, it follows that $J_h$ is a projection onto $X_h$. Furthermore, suppose that
  $T(u_1,\fu_2)=\trace_\Sigma u_1 \in \ran \trace_\Sigma|_{\cS^{k,\ell}(\cT_{t,x})}$.
  Then, $T J_h(u_1,\fu_2) = \trace_\Sigma\cJ_{\cS^{k,\ell}(\cT_{t,x})}u_1 = \trace_\Sigma u_1 =T(u_1,\fu_2)$.
  Finally, using the parabolic scaling, we estimate with Remark~\ref{rem1}, Lemma~\ref{lem:1}, and obvious bounds,
  \begin{align*}
    \| J_h(u_1,\fu_2) \|_X &\lesssim \| \cJ_{\cS^{k,\ell}(\cT_{t,x})}u_1 \|_{L^2(J;H^{1}(\Omega))}\\
    &\qquad + \| \cI_{\cP^{k-1}(\cT_t)\otimes RT^\ell(\cT_x)}\bigl(\fu_2-\nabla_x(-\Delta_x)^{-1}\partial_t(u_1-\cJ_{\cS^{k,\ell}(\cT_{t,x})}u_1)\bigr) \|_{L^2(Q)^d}\\
    &\qquad + \| \div\, J_h(u_1,\fu_2) \|_{L^2(Q)}\\
    &\lesssim \| u_1 \|_{L^2(J;H^{1}(\Omega))\cap H^1(J;H^{-1}(\Omega))}\\
    &\qquad + \| \fu_2 \|_{L^2(Q)^d} + \| \nabla_x(-\Delta_x)^{-1}\partial_t(u_1-\cJ_{\cS^{k,\ell}(\cT_{t,x})}u_1) \|_{L^2(Q)^d}\\
    &\qquad + \| \div\, (u_1,\fu_2) \|_{L^2(Q)}\\
    &\lesssim \| u_1 \|_{L^2(J;H^{1}(\Omega))\cap H^1(J;H^{-1}(\Omega))} + \| \fu_2 \|_{L^2(Q)^d} + \| \div\, (u_1,\fu_2) \|_{L^2(Q)},
  \end{align*}
  as well as
  \begin{align*}
    \| \partial_tu_1 \|_{L^2(J;H^{-1}(\Omega))} &\leq \| \div\, (u_1,\fu_2) \|_{L^2(J;H^{-1}(\Omega))} + \| \div_x\,\fu_2 \|_{L^2(J;H^{-1}(\Omega))}\\
    &\lesssim \| \div\, (u_1,\fu_2) \|_{L^2(Q)} + \| \fu_2 \|_{L^2(Q)^d}.
  \end{align*}
  We conclude that $J_h$ is uniformly bounded in $X$.
  Hence, Lemma~\ref{lemma:ihom} applies and yields quasi-optimality of $\bG_h$.
\end{proof}
\subsection{Isotropic Scott--Zhang operators}\label{sec:isoSZ}
We will collect the necessary ingredients for the construction put forward in~\cite{DieningST_23}.
For a conforming simplicial partition $\cT$ of $D\in\left\{ \Omega,\Gamma \right\}$ and Lagrange nodes $\cN^p(\cT)$ of $\cS^{p}(\cT)$,
let $(b_j)_{j\in\cN^p(\cT)}$ be the associated Bernstein basis, cf.~\cite{AinsworthAD_11}.
The support $\omega_j := \supp(b_j)$ consist of those elements $K\in\cT$ which share the Lagrange node $j$.
We will refer to $\omega_j$ as \textit{basis patch}.
According to~\cite[Prop.~2.14]{DieningST_23}, there exists a modified dual basis $(\psi_j)_{j\in\cN^p(\cT)}$
with $\psi_j\in\cS^{3p}_0(\cT)$ which fulfills
\begin{enumerate}
  \item[(i)] $\supp(\psi_j) \subseteq \omega_j = \supp(b_j)$,
  \item[(ii)] $\dual{b_j}{\psi_k}_D = \delta_{j,k}$,
  \item[(iii)] $\| b_j \|_{L^2(\omega_j)} \| \psi_j \|_{L^2(\omega_j)}\leq C_p$,
\end{enumerate}
where $C_p>0$ is a constant which depends only on $p$ and on the shape-regularity (or local quasi-uniformity if $d=1$) of $\cT$.
In particular, it holds
\begin{align}\label{eq:basis}
  \| b_{j} \|_{L^2(\omega_j)} \lesssim \diam(\omega_j)^{d/2} \qquad \text{ and } \qquad
  \| \psi_{j} \|_{L^2(\omega_j)} \lesssim \diam(\omega_j)^{-d/2},
\end{align}
where the first inequality is due to the definition of the Bernstein basis, and the second one is a consequence o
property $\rm{(iii)}$.
Above, $\dual{\cdot}{\cdot}_D$ is the $L^2(D)$-inner product, extended
to $H^{-1}(D)\times H^{1}_0(D)$.
If $D$ has a non-empty
boundary $\partial D$, we write $\cN^p(\cT) = \cN^p_0(\cT)\cup \cN^p_{\partial D}(\cT)$, where $\cN^p_0(\cT)$ denote the inner Lagrange nodes,
and $\cN^p_{\partial D}(\cT)$ the Lagrange nodes on the boundary $\partial D$.
The linear operators
\begin{align*}
  \cI_{\cS^p(\cT)}v &:= \sum_{j\in\cN^p(\cT)} \dual{v}{\psi_j}_D b_j,\\
  \cI_{\cS^p_0(\cT)}v &:= \sum_{j\in\cN^p_0(\cT)} \dual{v}{\psi_j}_D b_j,
\end{align*}
defined for $v\in H^{-1}(D)$, have the following properties, cf.~\cite[Thm.~2.1]{DieningST_23}.
\begin{theorem}\label{thm:dst}
  The linear operators $\cI_{\cS^p(\cT)}:H^{-1}(D)\to\cS^p(\cT)$ and
  $\cI_{\cS^p_0(\cT)}:H^{-1}(D)\to\cS^p_0(\cT)$ are projections onto $\cS^p(\cT)$ and $\cS^p_0(\cT)$, respectively.
  Furthermore, for $m=0,1$ and $K\in\cT$,
  \begin{align*}
    | \cI_{\cS^p(\cT)} u |_{H^m(K)} &\lesssim | u |_{H^m(\omega_K)} \qquad \text{ for all } u\in H^1(D),\\
    | \cI_{\cS^p_0(\cT)} u |_{H^m(K)} &\lesssim | u |_{H^m(\omega_K)} \qquad \text{ for all } u\in H^1_0(D),
  \end{align*}
  as well as
  \begin{align*}
    \| \cI_{\cS^p(\cT)} u \|_{H^{-1}(D)} &\lesssim \| u \|_{H^{-1}(D)} \qquad \text{ for all } u\in H^{-1}(D),\\
    \| \cI_{\cS^p_0(\cT)} u \|_{H^{-1}(D)} &\lesssim \| u \|_{H^{-1}(D)} \qquad \text{ for all } u\in H^{-1}(D).
  \end{align*}
  Hidden constants depend only on $p$ and shape-regularity (or local quasi-uniformity if $d=1$) of $\cT$.
\end{theorem}
We will also need the following inverse estimate.
\begin{lemma}\label{lem:invest}
  Let $\omega_j\subseteq\cT$ be a basis patch. Then, for every constant $c\in\bR$ it holds that
  \begin{align*}
    \| c \|_{L^2(\omega_j)} \lesssim h_{\omega_j}^{-1} \| c \|_{H^{-1}(\omega_j)},
  \end{align*}
  and the hidden constant depends only on shape-regularity of $\cT$.
\end{lemma}
\begin{proof}
  The analogous inverse estimate holds on every element $K\in\omega_j$ by~\cite[Lem.~1]{FuehrerHK_22}. A sum over all $\omega_j$ reveals
  \begin{align*}
    \| c \|_{L^2(\omega_j)}^2 = \sum_{K\in\omega_j} \| c \|_{L^2(K)}^2 \lesssim \sum_{K\in\omega_j} h_K^{-2} \| c \|_{H^{-1}(K)}^2
    \lesssim h_{\omega_j}^{-2} \| c \|_{H^{-1}(\omega_j)}^2,
  \end{align*}
  as shape regularity implies $h_{\omega_j}\simeq h_K$ as well as a bounded number of elements $K$ in $\omega_j$.
\end{proof}
\subsection{Anisotropic Scott--Zhang operators}\label{sec:anisotropic}
Let $\cT_t$ and $\cT_x$ be simplicial and admissible meshes of $J$ and $\Omega$, and $k,\ell\in\bN$.
Let $(b_{t,j})_{j\in\cN^k(\cT_t)}$, $(\psi_{t,j})_{j\in\cN^k(\cT_t)}$ be the Bernstein and modified dual bases for $\cN^k(\cT_t)$,
and $(b_{x,j})_{j\in\cN^\ell(\cT_x)}$, $(\psi_{x,j})_{j\in\cN^\ell(\cT_x)}$ 
be the Bernstein and modified dual bases for $\cN^\ell(\cT_x)$.
In order to apply the operators $\cI_{\cS^\ell(\cT_x)}$ and $\cI_{\cS^\ell_0(\cT_x)}$ defined in the preceding section in a space-time setting,
we will follow the arguments laid out in~\cite{DieningST_23,StevensonS_23}: applying the operator $\cI_{\cS^k(\cT_t)}$ everywhere
in space gives rise to a linear projection
\begin{align*}
  \cI_{\cS^k(\cT_t)}: L^2(J;H^{-1}(\Omega)) \to \cS^k(\cT_t)\otimes H^{-1}(\Omega),
\end{align*}
which fulfills the local stability property
\begin{align}\label{sec:anisotropic:time:stab}
  \| \partial_t^m\cI_{\cS^k(\cT_t)} v \|_{L^2(K_t;X)} \lesssim \| \partial_t^m v \|_{L^2(\omega_{K_t};X)}
\end{align}
for $m=0,1$ and $X=H^{-1}(\Omega)$ or $X=L^2(K_x)$, cf.~\cite[Thm.~4.3]{DieningST_23}.
The linear space-time interpolation operators
\begin{align*}
  \cI_{\cS^{k,\ell}(\cT_{t,x})}&:=\cI_{\cS^k(\cT_t)} \circ \cI_{\cS^\ell(\cT_x)}: L^2(J;H^{-1}(\Omega)) \to \cS^{k,\ell}(\cT_{t,x})\\
  \cI_{\cS^{k,\ell}_0(\cT_{t,x})}&:=\cI_{\cS^k(\cT_t)} \circ \cI_{\cS^\ell_0(\cT_x)}: L^2(J;H^{-1}(\Omega)) \to \cS^{k,\ell}_0(\cT_{t,x})
\end{align*}
are projections, and fulfill the local stability properties
\begin{align}\label{sec:anisotropic:stab}
  \begin{split}
    \| \partial_t^m \cI_{\cS^{k,\ell}_0(\cT_{t,x})} v \|_{L^2(K_t;L^2(K_x))} &\lesssim \| \partial_t^m v \|_{L^2(\omega_{K_t};L^2(\omega_{K_x}))},\\
    \| \partial_t^m \cI_{\cS^{k,\ell}_0(\cT_{t,x})} v \|_{L^2(K_t;H^{-1}(\Omega))} &\lesssim \| \partial_t^m v \|_{L^2(\omega_{K_t};H^{-1}(\Omega))},\\
    | \cI_{\cS^{k,\ell}_0(\cT_{t,x})} v |_{L^2(K_t;H^m(K_x))} &\lesssim | v |_{L^2(\omega_{K_t};H^m(\omega_{K_x}))}.
  \end{split}
\end{align}
for all $v\in H^1(J;H^{-1}(\Omega)) \cap L^2(J;H^1_0(\Omega))$, 
and $m=0,1$, analogously for $\cI_{\cS^{k,\ell}(\cT_{t,x})}$ and $v\in H^1(J;H^{-1}(\Omega)) \cap L^2(J;H^1(\Omega))$.
\begin{remark}\label{rem2}
  The spatial operator $\cI_{\cS^\ell(\cT_x)}$ (or $\cI_{\cS^\ell_0(\cT_x)}$ for that matter) as well as the temporal operator
  $\cI_{\cS^k(\cT_t)}$ are both defined using the dual bases from~\cite{DieningST_23} mentioned at the beginning of this Section.
  The fact that the modified dual basis functions share the support of the main basis functions implies that (local) stability
  on an element $K$ needs to take into account the original function on the bigger domain $\omega_K$. This can be seen
  in Theorem~\ref{thm:dst} for the spatial operator, and in~\eqref{sec:anisotropic:time:stab} for the temporal operator.
  It is possible to obtain reduce the domain of influence on the right-hand side for the temporal operator by using
  pointwise interpolation in time; this approach is used in~\cite{StevensonS_23}.
\end{remark}
The next result is a generalization of the Poincar\'e inequality, suited for parabolic problems.
In~\cite[Thm.~4]{StevensonS_23}, it is proved on a single space-time element. It extendeds to space-time patches as shown below.
\begin{lemma}\label{lem:parabolicpoincare}
  Let $\cT_t$ be a partition of $J$ into intervals and $\cT_x$ be conforming simplicial partition of $\Omega$.
  Let $\omega_{K_t}\subset\cT_t$ and $\omega_{K_x}\subset\cT_x$ be element patches.
  For $v \in L^2(\omega_{K_t};H^1(\omega_{K_x}))\cap H^1(\omega_{K_t};H^{-1}(\omega_{K_x}))$, it holds
  \begin{align*}
    \inf_{c\in\bR} \| v - c \|_{L^2(\omega_{K_t}\times\omega_{K_x})}
    \lesssim h_{K_x} \| \nabla_x v \|_{L^2(\omega_{K_t}\times\omega_{K_x})}
    + h_{K_t} h_{K_x}^{-1} \| \partial_t v \|_{L^2(\omega_{K_t};H^{-1}(\omega_{K_x}))}.
  \end{align*}
\end{lemma}
\begin{proof}
  The orthogonal projections
  $\Pi_{L^2(\omega_{K_x})}:L^2(\omega_{K_x})\to \cP^0(\omega_{K_x})$,
  $\Pi_{H^{-1}(\omega_{K_x})}:H^{-1}(\omega_{K_x})\to \cP^0(\omega_{K_x})$ and
  $\Pi_{L^2(\omega_{K_t})}:L^2(\omega_{K_t})\to\cP^0(\omega_{K_t})$ extend to bounded operators
  \begin{align*}
    \Pi_{L^2(\omega_{K_x})}&: L^2(\omega_{K_t}\times\omega_{K_x}) \to L^2(\omega_{K_t};\cP^0(\omega_{K_x})),\\
    \Pi_{H^{-1}(\omega_{K_x})}&: L^2(\omega_{K_t};H^{-1}(\omega_{K_x})) \to L^2(\omega_{K_t};\cP^0(\omega_{K_x})),\\
    \Pi_{L^2(\omega_{K_t})}&: L^2(\omega_{K_t}\times\omega_{K_x}) \to \cP^0(\omega_{K_t};L^2(\omega_{K_x})).
  \end{align*}
  We stress that due to our definition of patches, the sets $\cP^0(\omega_{K_x})$ and $\cP^0(\omega_{K_t})$ are globally constant functions
  on $\omega_{K_x}$ and $\omega_{K_t}$.
  Note that
  \begin{align*}
    v - \Pi_{H^{-1}(\omega_{K_x})}\Pi_{L^2(\omega_{K_t})}v =
    v - \Pi_{L^2(\omega_{K_x})}v + \Pi_{H^{-1}(\omega_{K_x})}
    \left( \Pi_{L^2(\omega_{K_x})}v - v + v - \Pi_{L^2(\omega_{K_t})}v \right).
  \end{align*}
  For any $w\in H^1(\omega_{K_x})$ there holds the Poincar\'e inequality
  \begin{align*}
    \| w - \Pi_{L^2(\omega_{K_x})}w \|_{L^2(\omega_{K_x})} \lesssim h_{K_x} \| \nabla_x w \|_{L^2(\omega_{K_x})}
  \end{align*}
  with a constant which depends only on shape-regularity. While this is an ubiquitous result in the case of single elements,
  the above version on patches also holds true and can be found, e.g., in~\cite{VeeserV_12}. By duality, this extends to
  \begin{align*}
    h_{K_x}^{-1}\|  w - \Pi_{L^2(\omega_{K_x})}w \|_{H^{-1}(\omega_{K_x})} \lesssim \| w - \Pi_{L^2(\omega_{K_x})}w \|_{L^2(\omega_{K_x})}
    \lesssim h_{K_x} \| \nabla_x w \|_{L^2(\omega_{K_x})}.
  \end{align*}
  The last estimate, the triangle inequality, the inverse estimate from Lemma~\ref{lem:invest}, the continuity of $\Pi_{H^{-1}(\omega_{K_x})}$,
  and the appromxation result~\cite[Thm.~4.3]{DieningST_23} for $\Pi_{L^2(\omega_{K_t})}$ imply
  \begin{align*}
    \inf_{c\in\bR} \| v - c \|_{L^2(\omega_{K_t}\times\omega_{K_x})} &\lesssim
    \| v - \Pi_{L^2(\omega_{K_x})}v \|_{L^2(\omega_{K_t}\times\omega_{K_x})}
    + h_{K_x}^{-1} \| \Pi_{L^2(\omega_{K_x})}v - v + v - \Pi_{L^2(\omega_{K_t})}v \|_{L^2(\omega_{K_t};H^{-1}(\omega_{K_x}))}\\
    &\lesssim 
    h_{K_x} \| \nabla_x v \|_{L^2(\omega_{K_t}\times\omega_{K_x})}
    + h_{K_x}^{-1}\| v - \Pi_{L^2(\omega_{K_t})}v \|_{L^2(\omega_{K_t};H^{-1}(\omega_{K_x}))}\\
    &\lesssim h_{K_x} \| \nabla_x v \|_{L^2(\omega_{K_t}\times\omega_{K_x})}
    + h_{K_t} h_{K_x}^{-1} \| \partial_t v \|_{L^2(\omega_{K_t};H^{-1}(\omega_{K_x}))},
  \end{align*}
  and the proof is finished.
\end{proof}
\subsection{Scott--Zhang-type space-time interpolation operator}
In this section, we construct the operator $\cJ_{\cS^{k,\ell}(\cT_{t,x})}$ from Theorem~\ref{thm:main}.
First, we let $\cT_x^\Gamma$ be the simplicial mesh on $\Gamma$ induced by $\cT_x$.
Given the Berstein basis $(b_{x,j})_{j\in\cN^\ell(\cT_x)}$ of $\cS^\ell(\cT_x)$, the traces
$(\trace_\Gamma\,b_{x,j})_{j\in\cN^\ell_\Gamma(\cT_x)}$ constitute a Bernstein basis for
$\trace_\Gamma\,\cS^\ell(\cT_x) = \cS^\ell(\cT_x^\Gamma)$. Denote by $(\psi_{x,j}^\Gamma)_{j\in\cN^\ell_\Gamma(\cT_x)}$
the corresponding modified dual basis from Section~\ref{sec:isoSZ}. Define the linear operator
\begin{align*}
  \cJ_{\cS_\Sigma^{k,\ell}(\cT_{t,x})}: L^2(J;H^1(\Omega)) \cap H^1(J;H^{-1}(\Omega))\to\cS^{k,\ell}(\cT_{t,x})
\end{align*}
by
\begin{align*}
  \cJ_{\cS_\Sigma^{k,\ell}(\cT_{t,x})} v :=
  \sum_{\substack{m\in\cN^k(\cT_t)\\n\in\cN_\Gamma^\ell(\cT_x)}} b_{t,m}\otimes b_{x,n} \int_0^T \int_\Gamma \psi_{t,m}(s)\psi_{x,n}^\Gamma(y)
  v(s,y)\,dy\,ds.
\end{align*}
This operator is well defined, as traces of functions in $L^2(J;H^1(\Omega)) \cap H^1(J;H^{-1}(\Omega)$ belong
to $H^{1/2,1/4}(\Sigma)$. Basis patches for the temporal mesh $\cT_t$ are denoted by $\omega_{t,m}$, while basis patches
for the spatial mesh $\cT_x$ are denoted by $\omega_{x,n}$. Additionally, basis patches for the mesh $\cT_{x}^\Gamma$ are denoted
by $\omega_{x,n}^\Gamma$. It follows from~\eqref{eq:basis} that
\begin{align}\label{eq:basis:bdry}
  \| \psi_{x,j}^\Gamma \|_{L^2(\omega_{x,j}^\Gamma)} \lesssim \diam(\omega_{x,j}^\Gamma)^{-(d-1)/2}.
\end{align}
In the following, we will prove Theorem~\ref{thm:main}.
\begin{proof}[Proof of Theorem~\ref{thm:main}]
  Define the operator $\cJ_{\cS^{k,\ell}(\cT_{t,x})}$ by
  \begin{align*}
    \cJ_{\cS^{k,\ell}(\cT_{t,x})} := \cI_{\cS_0^{k,\ell}(\cT_{t,x})} + \cJ_{\cS_\Sigma^{k,\ell}(\cT_{t,x})}.
  \end{align*}
  We will prove the stated properties.
  \begin{itemize}
    \item[(ii)] 
      First, it holds $\trace_\Sigma\,\cJ_{\cS^{k,\ell}(\cT_{t,x})} = \trace_\Sigma\,\cJ_{\cS_\Sigma^{k,\ell}(\cT_{t,x})}$.
      Note that $\trace_\Sigma\,\cS^{k,\ell}(\cT_{t,x}) = \cS^k(\cT_t)\otimes \cS^\ell(\cT_x^\Gamma)$.
      Hence, if $\trace_\Sigma v\in\trace_\Sigma \cS^{k,\ell}(\cT_{t,x})$, then
      \begin{align*}
        \trace_\Sigma\,v = \sum_{\substack{m\in\cN^k(\cT_t)\\n\in\cN^\ell_\Gamma(\cT_x)}}v_{m,n} b_{t,m}\trace_\Gamma\,b_{x,n}
      \end{align*}
      Then, given that $\psi_{t,m}$ and $\psi_{x,n}^\Gamma$ are dual to $b_{t,m}$ and $\trace_\Gamma\,b_{x,n}$, respectively, this yields
      \begin{align*}
        \trace_\Sigma\,\cJ_{\cS^{k,\ell}(\cT_{t,x})}v &= \trace_\Sigma\,\cJ_{\cS_\Sigma^{k,\ell}(\cT_{t,x})}v\\
        &= \sum_{\substack{m\in\cN^k(\cT_t)\\n\in\cN_\Gamma^\ell(\cT_x)}}
        b_{t,m} \trace_\Gamma\,b_{x,n} \int_0^T \int_\Gamma \psi_{t,m}(s)\psi_{x,n}^\Gamma(y)
        \trace_\Sigma \,v(s,y)\,dy\,ds\\
        &=  \sum_{\substack{m\in\cN^k(\cT_t)\\n\in\cN_\Gamma^\ell(\cT_x)}} v_{m,n} b_{t,m} \trace_\Gamma\,b_{x,n}.
      \end{align*}
    \item[(i)] 
      To see that $\cJ_{\cS^{k,\ell}(\cT_{t,x})}$ is projection, let $v\in \cS^{k,\ell}(\cT_{t,x})$ and write
      \begin{align*}
        v = \sum_{m\in\cN^k(\cT_t)}\sum_{n\in\cN^\ell_0(\cT_x)} v_{m,n} b_{t,m} b_{x,n}
        + \sum_{m\in\cN^k(\cT_t)}\sum_{n\in\cN^\ell_\Gamma(\cT_x)} v_{m,n} b_{t,m} b_{x,n} = v_0 + v_\Sigma.
      \end{align*}
      Then, as $b_{x,j}$ and $\psi_{x,j}$ as well as $b_{t,i}$ and $\psi_{t,i}$ are dual bases,
      we see $\cI_{\cS_0^{k,\ell}(\cT_{t,x})}v_0=v_0$ and $\cI_{\cS_0^{k,\ell}(\cT_{t,x})}v_\Sigma = 0$.
      As $\trace_\Sigma\,v_0=0$, also $\cJ_{\cS_\Sigma^{k,\ell}(\cT_{t,x})}v_0=0$. 
      \begin{align*}
        \cJ_{\cS^{k,\ell}(\cT_{t,x})}v = \cI_{\cS_0^{k,\ell}(\cT_{t,x})}v + \cJ_{\cS_\Sigma^{k,\ell}(\cT_{t,x})}v = v_0 + \cJ_{\cS_\Sigma^{k,\ell}(\cT_{t,x})}v_\Sigma.
      \end{align*}
      Arguing as above, we get $\cJ_{\cS_\Sigma^{k,\ell}(\cT_{t,x})}v_\Sigma = v_\Sigma$.
    \item[(iii)]
      We will show the three inequalities separately.
      \begin{itemize}
	\item Due to the stability bounds~\eqref{sec:anisotropic:stab}, it suffices to consider the operator $\cJ_{\cS_\Sigma^{k,\ell}(\cT_{t,x})}$.
          Note that for $c\in\bR$ it holds
	  \begin{align*}
	    \cJ_{\cS_\Sigma^{k,\ell}(\cT_{t,x})} c &= 
	    c \sum_{n\in\cN^\ell_\Gamma(\cT_x)} b_{x,n} \int_\Gamma \psi_{x,n}^\Gamma(y)\,dy
	    \sum_{m\in\cN^k(\cT_t)} b_{t,m} \int_0^T \psi_{t,m}(s)\,ds\\
	    &= c \sum_{n\in\cN^\ell_\Gamma(\cT_x)} b_{x,n} \int_\Gamma \psi_{x,n}^\Gamma(y)\,dy
	  \end{align*}
	  due to duality of $b_{t,m}$ and $\psi_{t,m}$,
	  and therefore $\partial_t \cJ_{\cS_\Sigma^{k,\ell}(\cT_{t,x})} c = 0$.
	  For an element $K_x\in\cT_x$, we conclude with~\eqref{eq:basis} and~\eqref{eq:basis:bdry}
	  \begin{align*}
            \|\partial_t\cJ_{\cS_\Sigma^{k,\ell}(\cT_{t,x})} v \|_{L^2(K_x)}
	    &= \|\partial_t\cJ_{\cS_\Sigma^{k,\ell}(\cT_{t,x})} (v - c_{K_x}) \|_{L^2(K_x)}\\
            &\lesssim \sum_{\substack{m\in\cN^k(\cT_t)\\n\in\cN_\Gamma^\ell(\cT_x)\cap \overline{K_x}}} |\partial_t b_{t,m}| \| b_{x,n} \|_{L^2(K_x)}
	    \| \psi_{t,m} \|_{L^2(\omega_{t,m})} \| \psi_{x,n}^\Gamma \|_{L^2(\omega_{x,n}^\Gamma)}
	    \| v - c_{K_x} \|_{L^2(\omega_{t,m}\times\omega_{x,n}^\Gamma)}\\
            &\lesssim\sum_{\substack{m\in\cN^k(\cT_t)\\n\in\cN_\Gamma^\ell(\cT_x)\cap \overline{K_x}}} h_{K_x}^{1/2} |\partial_tb_{t,m}|
	    \| \psi_{t,m} \|_{L^2(\omega_{t,m})} \| v - c_{K_x} \|_{L^2(\omega_{t,m}\times\omega_{x,n}^\Gamma)}.
	  \end{align*}
	  We calculate for $w\in H^1_0(\Omega)$
          \begin{align*}
            |\vdual{\partial_t\cJ_{\cS_\Sigma^{k,\ell}(\cT_{t,x})}v}{w}_\Omega|
            &\leq \sum_{\substack{K_x\in\cT_x\\\overline{K_x}\cap\Gamma\neq\emptyset}} \| w \|_{L^2(K_x)}
            \|\partial_t\cJ_{\cS_\Sigma^{k,\ell}(\cT_{t,x})} (v - c_{K_x}) \|_{L^2(K_x)}\\
            &\lesssim \sum_{\substack{K_x\in\cT_x\\\overline{K_x}\cap\Gamma\neq\emptyset}} h_{K_x} \| \nabla_x w \|_{L^2(\omega_{K_x})}
            \sum_{\substack{m\in\cN^k(\cT_t)\\n\in\cN_\Gamma^\ell(\cT_x)\cap \overline{K_x}}} h_{K_x}^{1/2} |\partial_tb_{t,m}|
	    \| \psi_{t,m} \|_{L^2(\omega_{t,m})} \| v - c_{K_x} \|_{L^2(\omega_{t,m}\times\omega_{x,n}^\Gamma)},
          \end{align*}
          where in the last step we used a local Poincar\'e inequality together with the fact that $K_x$ lies next to
          the boundary $\Gamma$, where $w$ vanishes.
          The Cauchy--Schwarz inequality for sums and integration in time as well as an inverse inequality, we obtain
	  again with~\eqref{eq:basis}
          \begin{align*}
            \| \partial_t \cJ_{\cS_\Sigma^{k,\ell}(\cT_{t,x})}v \|_{L^2(K_t;H^{-1}(\Omega))}^2
            &\lesssim \sum_{\substack{K_x\in\cT_x\\\overline{K_x}\cap\Gamma\neq\emptyset}}
            \sum_{\substack{m\in\cN^k(\cT_t)\cap\overline{K_t}\\n\in\cN_\Gamma^\ell(\cT_x)\cap \overline{K_x}}}
	    h_{K_x}^3 \| \partial_t b_{t,m} \|_{L^2(K_t)}^2 \| \psi_{t,m}
	    \|_{L^2(\omega_{t,m})}^2\| v - c_{K_x} \|_{L^2(\omega_{t,m}\times\omega_{x,n}^\Gamma)}^2\\
            &\lesssim 
            \sum_{\substack{K_x\in\cT_x\\\overline{K_x}\cap\Gamma\neq\emptyset}}
	    \sum_{n\in\cN_\Gamma^\ell(\cT_x)\cap \overline{K_x}} h_{K_x}^3 h_{K_t}^{-2}\| v - c_{K_x} \|_{L^2(\omega_{K_t}\times\omega_{x,n}^\Gamma)}^2\\
            &\lesssim 
            \sum_{\substack{K_x\in\cT_x\\\overline{K_x}\cap\Gamma\neq\emptyset}}
            h_{K_x}^2 h_{K_t}^{-2}\| v - c_{K_x} \|_{L^2(\omega_{K_t}\times\omega_{K_x})}^2 + h_{K_x}^4 h_{K_t}^{-2} \| \nabla_x v \|_{L^2(\omega_{K_t}\times\omega_{K_x})}^2,
          \end{align*}
          where we used a weighted trace inequality in spatial direction in the last step.
	  The parabolic Poincar\'e inequality Lemma~\eqref{lem:parabolicpoincare} yields
          \begin{align*}
            \| \partial_t \cJ_{\cS_\Sigma^{k,\ell}(\cT_{t,x})}v \|_{L^2(K_t;H^{-1}(\Omega))}^2
            \lesssim \sum_{\substack{K_x\in\cT_x\\\overline{K_x}\cap\Gamma\neq\emptyset}}
            h_{K_x}^4 h_{K_t}^{-2} \| \nabla_x v \|_{L^2(\omega_{K_t}\times\omega_{K_x})}^2 + \| \partial_t v \|_{L^2(\omega_{K_t}\times H^{-1}(\omega_{K_x}))}^2,
          \end{align*}
          which shows the first of the stipulated estimates.
	\item To prove the second estimate, we calculate using~\eqref{eq:basis} and~\eqref{eq:basis:bdry}
          \begin{align*}
            \| \cJ_{\cS^{k,\ell}(\cT_{t,x})} v \|_{L^2(K_t\times K_x)}
            &\leq 
            \sum_{\substack{m\in\cN^k(\cT_t)\cap\overline{K_t}\\n\in\cN_\Gamma^\ell(\cT_x)\cap \overline{K_x}}}
	    \| b_{t,m} \|_{L^2(K_t)} \| b_{x,n} \|_{L^2(K_x)} \| \psi_{t,m} \|_{L^2(\omega_{t,m})} \| \psi_{x,n}^\Gamma \|_{L^2(\omega_{x,n}^\Gamma)}
	    \| v \|_{L^2(\omega_{t,m};L^2(\omega_{x,n}))}\\
            &\lesssim 
            \sum_{\substack{m\in\cN^k(\cT_t)\cap\overline{K_t}\\n\in\cN_\Gamma^\ell(\cT_x)\cap \overline{K_x}}}
	    h_{K_x}^{1/2} \| v \|_{L^2(\omega_{t,m};L^2(\omega_{x,n}))}\\
            &\lesssim \| v \|_{L^2(\omega_{K_t}\times \omega_{K_x})} + h_{K_x} \| \nabla_x v \|_{L^2(\omega_{K_t}\times \omega_{K_x})},
          \end{align*}
          where we applied a weighted trace estimate in spatial direction in the last step.
        \item To prove the third estimate, note that
          \begin{align*}
            1 = \sum_{n\in\cN_0^\ell(\cT_x)} b_{x,n}\vdual{1}{\psi_{x,n}}_\Omega + 
            \sum_{n\in\cN_\Gamma^\ell(\cT_x)} b_{x,n}\vdual{1}{\psi_{x,n}^\Gamma}_\Gamma,
          \end{align*}
          which implies for $c\in\bR$ that
          \begin{align*}
            \cJ_{\cS^{k,\ell}(\cT_{t,x})} c = \cI_{\cS_0^{k,\ell}(\cT_{t,x})} c  + \cJ_{\cS_\Sigma^{k,\ell}(\cT_{t,x})} c=c.
          \end{align*}
	  The inverse estimate $\| \nabla_x b_{x,n} \|_{L^2(K_x)} \lesssim h_{K_x}^{-1} \| b_{x,n} \|_{L^2(K_x)}$
	  together with~\eqref{eq:basis} and~\eqref{eq:basis:bdry} imply
          \begin{align*}
            \| \nabla_x \cJ_{\cS^{k,\ell}(\cT_{t,x})} v \|_{L^2(K_t\times K_x)} &= 
            \| \nabla_x \cJ_{\cS^{k,\ell}(\cT_{t,x})} ( v - c) \|_{L^2(K_t\times K_x)}\\
            &\leq 
	    \sum_{n\in\cN_0^\ell(\cT_x)\cap \overline{K_x}} \| \nabla_x b_{x,n} \|_{L^2(K_x)} \| \psi_{x,n}\|_{L^2(\omega_{x,n})}
	    \| v - c \|_{L^2(\omega_{K_t}\times \omega_{x,n})} + \\
	    &\qquad\sum_{n\in\cN_\Gamma^\ell(\cT_x)\cap \overline{K_x}} \| \nabla_x b_{x,n} \|_{L^2(K_x)} \| \psi_{x,n}^\Gamma\|_{L^2(\omega_{x,n}^\Gamma)}
	    \| v - c \|_{L^2(\omega_{K_t}\times \omega_{x,n}^\Gamma)}\\
	    &\lesssim \sum_{n\in\cN_0^\ell(\cT_x)\cap \overline{K_x}} h_{K_x}^{-1} \| v - c \|_{L^2(\omega_{K_t}\times \omega_{x,n})} +\\
	    &\qquad\sum_{n\in\cN_\Gamma^\ell(\cT_x)\cap \overline{K_x}} h_{K_x}^{-1/2} \| v - c \|_{L^2(\omega_{K_t}\times\omega_{x,n}^{\Gamma})}\\
            &\lesssim h_{K_x}^{-1} \| v - c \|_{L^2(\omega_{K_t}\times\omega_{K_x})} + \| \nabla_x v \|_{L^2(\omega_{K_t}\times\omega_{K_x})},
          \end{align*}
          where we applied a weighted trace estimate in spatial direction in the last step.
	  The parabolic Poincar\'e inequality Lemma~\eqref{lem:parabolicpoincare} implies the desired result.
      \end{itemize}
  \end{itemize}
\end{proof}
\begin{remark}\label{rem3}
  In the context of Section~\ref{sec:appl}, if $E:H^{1/2,1/4}(\Sigma)\to L^2(H^1(\Omega))\cap H^1(H^{-1}(\Omega))$ denotes an arbitrary,
  not necessarily bounded extension operator, then the operator
  $\trace_\Sigma J_{\cS_\Sigma^{k,\ell}(\cT_{t,x})} E : H^{1/2,1/4}(\Sigma)\to \ran \trace_\Sigma|_{\cS^{k,\ell}(\cT_{t,x})}$ is a valid
  candidate for the operator $P_h$ approximating Dirichlet boundary conditions.
\end{remark}
\bibliographystyle{abbrv}
\bibliography{literature}
\end{document}